\documentclass[12pt]{article}
\usepackage{amsmath}
\usepackage{amsfonts}
\usepackage{amsthm}
\usepackage{amssymb, setspace}
\usepackage{color,graphicx, epsfig, geometry, hyperref,fancyhdr}
\newtheorem{theorem}{Theorem}[section]
\usepackage{framed}
\usepackage{float} 
\newtheorem{definition}{Definition}[section] 
\newtheorem{corollary}{Corollary}[section] 
\newtheorem{lemma}{Lemma}[section]
\newtheorem{remark}{Remark}[section] 
\newcommand{\N}{\mathbb{N}}

\newcommand{\R}{\mathbb{R}}
\newcommand{\C}{\mathbb{C}}

\newcommand{\dist}{\textrm{dist}}

\newcommand{\grad}{\nabla}

\newcommand{\eps}{\varepsilon}

\graphicspath{{pics/}}

\begin{document}
\setlength{\parskip}{1mm}
\setlength{\oddsidemargin}{0.1in}
\setlength{\evensidemargin}{0.1in}
\lhead{}
\rhead{}
 
\begin{center}
{\bf \Large  The Imaging of Small Perturbations in an Anisotropic Media} \\
\vspace{0.2in}
Fioralba Cakoni,\footnote{Department of Mathematics, Rutgers University, Piscataway, NJ 08854-8019, USA. E-mail: iharris@math.tamu.edu}
Isaac Harris\footnote{Department of Mathematics, Texas A$\&$M  University, College Station, TX 77843-3368, USA. E-mail: iharris@math.tamu.edu}
and Shari Moskow\footnote{Department of Mathematics, Drexel University, Philadelphia, PA  19104-2875, USA. E-mail: moskow@math.drexel.edu}
\end{center}

\begin{abstract}
In this paper, we employ asymptotic analysis to determine information about small volume defects in a known anisotropic scattering medium from far field scattering data. The location of the defects is reconstructed via the MUSIC algorithm from  the range of the multi-static response matrix derived from the asymptotic expansion of the far field pattern in the presence of small defects. Since the same data determines the transmission eigenvalues corresponding to the perturbed media, we investigate how the presence of the defects changes the transmission eigenvalues and use this information to recover the strength of the small defects.  We provide convergence results on  transmission eigenvalues as the size of the defects tends to zero as well as derive the first correction term in the asymptotic expansion of the simple  transmission eigenvalues.  Numerical examples are presented to show the viability of our imaging method.
\end{abstract}
 

\section{Introduction}\label{intro}
The imaging of anisotropic media from scattering data is a challenging problem mainly due to the non-uniqueness issue \cite{nonuniq}. Yet, in many applications in medical imaging and non-destructive testing, the scattering media exhibit anisotropic properties in the interaction with probing waves. The so-called qualitative methods in inverse scattering \cite{p1} provide imaging techniques to obtain information on changes in material properties of a known anisotropic media. This work concerns the imaging of  small volume (possibly anisotropic) perturbations  of a known anisotropic inhomogeneous media in acoustic wave propagation (for the case of ${\mathbb R}^3$) or specially polarized electromagnetic wave propagation (for the case of ${\mathbb R}^2$). Combining asymptotic analysis with MUSIC and the related transmission eigenvalue problem we derive a range test for the location of  small perturbations and computable formulas that provide information about the strength (involving the contrast and geometrical features) of the small perturbation. There is a vast literature on the MUSIC algorithm for a variety of scattering problems  \cite{amari}, \cite{ffformula}, \cite{park} and we recall here its formulation for the anisotropic inhomogeneous media. The asymptotic analysis of the transmission eigenvalue problem for isotropic media is  studied in \cite{CM} and \cite{CMR}. One of the main contributions of this study is the asymptotic analysis of the transmission eigenvalue problem for anisotropic media with the first order correction term for the perturbation of the  eigenvalues. Note that the transmission eigenvalue problem is non-linear and non-selfadjoint, and the mathematical structure of this problem for anisotropic media is different from the isotropic case. In addition, we show  how to use the asymptotic expansion for the perturbation of transmission eigenvalues together with the MUSIC algorithm to image small volume perturbations of anisotropic media.

Let us now precisely formulate the  problem under consideration. To this end let $D \subset \R^d$ (for $d=2$ or 3) be a bounded domain with piecewise smooth boundary which denotes the support of the anisotropic media to be tested.  The real valued  symmetric matrix $A(x) \in {C}^1(D, \R^{d \times d})$ with  smooth entries and the smooth function $n\in  {C}^1(D)$ represent the constitutive parameters for the unperturbed (``healthy'') anisotropic  media. Without loss of generality we assume that outside the scatterer $D$ the background media has  refractive index scaled to one, i.e. $A(x)=I$ and $n(x)=1$  in $x \in \R^d \setminus \overline{D}$, where $I$ denotes the identity matrix. We define 
$$  { A}_b(x) =\left\{ \begin{array}{rl} I \;\; & \quad \, x \in \R^{d} \setminus \overline{D}  \\ A(x) & \quad \, x \in D  \end{array}\right.  \quad \text{ and } \quad {n}_b(x) =\left\{ \begin{array}{rl} 1 \;\;\; & \quad \, x \in \R^d \setminus \overline{D}  \\ n(x) & \quad \, x \in D.  \end{array}\right. $$ 
Now the scattering  of a time harmonic incident plane wave $\text{e}^{\text{i}kx \cdot \hat y}$ with incident direction $ \hat y \in \mathbb{S}$  by the unperturbed media  (i.e. without defects) is mathematically formulated as: find $u_b \in H^1_{loc}(\R^d)$ with $u_b =u^s_b + \text{e}^{\text{i}kx \cdot \hat y}$  such that
\begin{eqnarray}
\grad \cdot { A}_b (x) \grad u_b +k^2 {n}_b (x)  u_b=0 \,  &\textrm{ in }& \,  \R^d  \label{musicprod1}\\
\lim\limits_{r \rightarrow \infty} r^{\frac{d-1}{2}} \left( \frac{\partial u^s_b}{\partial r} -iku^s_b \right)=0 \label{2}, \label{musicprod2}
\end{eqnarray}
where  $\mathbb {S}$ denotes the unit circle/sphere, $r=|x|$,  and the Sommerfeld radiation condition (\ref{2}) is satisfied uniformly with respect to $\hat x=x/|x|$.
Here $u_b$ is the total field  in  the background (including the homogeneous part and the media of compact support $\overline{D}$)  and $u^s_b$ is the scattered field  due to the region $D$. 
Recall that the scattered radiating field $u_b^s(\cdot, \hat y )$, which depends on the incident direction $\hat y$, has the following asymptotic expansion \cite{coltonkress}
\begin{equation}\label{eq1}
u^s_b(x,\hat y)=\frac{\text{e}^{ik|x|}}{|x|^{\frac{d-1}{2}}} \left\{u_b^{\infty}(\hat{x}, \hat y ) + \mathcal{O} \left( \frac{1}{|x|}\right) \right\}\; \textrm{  as  } \;  |x| \to \infty
\end{equation}
where $\hat x:=x/|x|$, and $u_b^{\infty}(\hat x, \hat y) $, which depends on the incident direction $\hat y$ and observation direction $\hat{x}$, is the corresponding far field pattern.  Now we consider the small defective regions that are given by $z_m + \eps B_m$ where $B_m$ is a smooth deformation of a ball centered at the origin. Let $A_m$ and $n_m$ be constant constitutive parameters for the defective regions given by $z_m + \eps B_m$ and  assume that 
$$|z_i-z_j | \geq c_0 >0 \quad \text {for all } \, \, \, i \neq j  \quad \text {with } \, \, \,  i,j= 1, 2, \dots M \qquad \mbox{and}$$
 $$\dist (z_m \, , \partial D) \geq c_0 >0 \quad \text {for all } \, \, \, m= 1, 2, \dots M.$$
\noindent The union of the defective regions is denoted by $D_\eps = \bigcup\limits_{m=1}^{M} (z_m + \eps B_m)$ and we let 
$$  {A}_\eps (x) =\left\{ \begin{array}{rl} A_m & \quad  \, x \in (z_m + \eps B_m) \\  {A}(x) & \quad  \, x \in \R^d \setminus \overline{D}_\eps  \end{array}\right.   \quad  \mbox{and} \quad  {n}_\eps (x) =\left\{ \begin{array}{rl} n_m & \quad  \, x \in (z_m + \eps B_m) \\ {n}(x) & \quad \, x \in \R^d \setminus \overline{D}_\eps. \end{array}\right.  $$ 
The scattering problem for the media with the defective region $D_\eps$ now reads: find $u_\eps \in H^1_{loc}(\R^d)$ with $u_\eps =u_\eps^s + \text{e}^{\text{i}kx \cdot \hat y }$ such that
\begin{eqnarray}
\grad \cdot {A}_\eps (x) \grad u_\eps +k^2 {n}_\eps (x) u_\eps=0 \,  &\textrm{ in }& \,  \R^d  \label{musicprod3}\\
\lim\limits_{r \rightarrow \infty} r^{\frac{d-1}{2}} \left( \frac{\partial u_\eps^s}{\partial r} -ik u_\eps^s \right)=0. \label{musicprod4}
\end{eqnarray}
Similarly since $u^s_\eps$ is a radiating solution to the Helmholtz equation in $\R^d \setminus \overline{D}$,  it assumes a similar asymptotic expansion as (\ref{eq1}), and we denote by $u_\eps^{\infty}(\hat{x}, \hat y)$ its corresponding far field pattern. 
In this study we assume that the media is non-absorbing, and  $\inf_{x \in D} n(x)=n_0>0$, $n_m>0$, and 
\begin{eqnarray}
\inf_{x \in D} \inf_{|\xi|=1} \overline{\xi} \cdot A(x) \xi =A_{min}>0 \, \, &\textrm{ and }& \, \, \sup_{x \in D} \sup_{|\xi|=1} \overline{\xi} \cdot A(x) \xi =A_{max}< \infty  \label{ca1}
\end{eqnarray} 
For later use let us denote
\begin{eqnarray}
\min_{m=1 \dots M}\inf_{|\xi|=1} \overline{\xi} \cdot A_m \xi =a_{min}>0 \, \, &\textrm{ and }& \, \,  \max_{m=1 \dots M}\sup_{|\xi|=1} \overline{\xi} \cdot A_m \xi =a_{max}< \infty.  \label{ca2}
\end{eqnarray} 
The {\it{inverse problem}} we consider here is to determine the location $\{z_m\}_{m=1,M}$ of the perturbations and information about $A_m$ and $n_m$ from knowledge of $u_\eps^{\infty}(\hat{x}, \hat y)$  for several  $\hat x, \hat y \in {\mathbb S}$, provided that $A_b(x)$ and $n_b(x)$ are known.

In general, the support $D_\epsilon$ of the defects can be determined from the {\it far field operator} 
\begin{equation}\label{ffo}
(Fg)(\hat x)=\int_{\mathbb S}\left[u_\eps^{\infty}(\hat{x}, \hat y) -u_b^{\infty}(\hat{x}, \hat y)\right]g(\hat y)\,d\hat y\qquad \qquad \hat x\in{\mathbb S}
\end{equation}
via the factorization method \cite{CakoniHarris}. In addition,  it is well-known \cite{p1} that the far field operator $F$ determines the {\it real transmission eigenvalues}  which are defined below.
\begin{definition} Transmission eigenvalues are the values $k_\eps \in \C$ for which there is a non-trivial solution $(w,v) \in H^1(D) \times H^1(D)$ of 
\begin{eqnarray}
\grad \cdot A_\eps \grad w +k_\eps^2 n_\eps w=0  \quad \text{and } \quad  \Delta v + k_\eps^2 v=0 \, \, && \textrm{ in } \,  D \label{tedefect1}\\
  w=v \quad \text{and } \quad \frac{\partial w}{\partial \nu_{A_\eps}}=\frac{\partial v}{\partial \nu} \quad  & & \textrm{ on } \partial D. \label{tedefect2}
\end{eqnarray}
\end{definition} 
These transmission eigenvalues can be used to obtain information about $A_\epsilon$ and $n_\epsilon$ \cite{p1}. In this paper we will make use of the small volume feature of the defects and use asymptotic analysis to determine the locations $z_m$, $m=1 \dots M$  of the small inhomogeneities. Then, based on the perturbation formulas of the transmission eigenvalues, we determine geometric and physical information about these small defects via polarization tensors in the asymptotic formulas.
\section{Asymptotic Formulas and the MUSIC Algorithm}
To avoid technical difficulties with asymptotic expansions,  without loss of generality we assume that the anisotropic media is homogeneous, i.e. the matrix $A$ and the scalar $n$ are constant. We derive the multi-static response matrix by exploiting the fact that the each of the defective regions has small volume as in \cite{park}, which will be used to reconstruct the defective regions.  The multi-static response matrix can be seen as the discrete version of the far field operator $F$ given by (\ref{ffo}).  To this end we first  recall $\mathbb{G}( \cdot ,\cdot)$ the Green's function for the background layered media, i.e. the solution of
\begin{eqnarray*}
&&\grad \cdot {A}(x) \grad \mathbb{G}( \cdot ,z) +k^2 {n}(x)  \mathbb{G}( \cdot ,z)=-\delta (\, \cdot \, - z) \, \,  \textrm{ in } \,  \R^d \\
&&\lim\limits_{r \rightarrow \infty} r^{\frac{d-1}{2}} \left( \frac{\partial  \mathbb{G}( \cdot ,z)}{\partial r} -ik  \mathbb{G}( \cdot ,z) \right)=0.  
\end{eqnarray*}
Let $\mathbb{G}^{\infty}( \cdot \, , z) \in L^2(\mathbb{S})$ be it's far field pattern. Since $A(x)$ is a symmetric constant positive definite matrix for $x \in D$ and $n(x)$ is a positive constant for $x \in D$ we have by Theorem 5.1 in  \cite{CakoniHarris} that 
\begin{equation}\label{mix}
\mathbb{G}^{\infty}(\hat{x},z)={\gamma} u_b(z,-\hat{x}) \quad \text{ where } \quad \gamma = \frac{\text{e}^{i \pi/4}}{\sqrt{8 \pi k}} \,\,\text{ in } \, \, \R^2 \, \, \, \,  \text{ and }\quad    \gamma = \frac{1}{4 \pi } \, \, \text{ in } \, \,  \R^3.
\end{equation}
It can be shown (see \cite{ffformula}) by using Green's identities and the Sommerfeld radiation condition that $u_\eps$ satisfies the Lippmann-Schwinger representation formula given by 
\begin{eqnarray}
&&u_\eps (x,\hat y)=u_b(x,\hat y)+ \sum\limits_{m=1}^{M } k^2  \int\limits_{z_m + \eps B_m} \hspace{-0.2cm}(n_m-n ) \mathbb{G}( x , z) u_\eps(z,\hat y) \, dz \nonumber \\
&&\hspace{2in} + \int\limits_{z_m + \eps B_m} \hspace{-0.2cm} (A-A_m)\grad_z \mathbb{G}( x , z) \cdot \grad u_\eps(z,\hat y) \, dz. \label{intrep}
\end{eqnarray}
By linearity it is clear that the scattered field $u^s=u_\eps^s-u_b^s$ is due to the defective regions $D_\eps$ and  from \eqref{intrep} an asymptotic expansion for $u^{\infty}(\hat{x}, \hat y )$ can be obtained by combining the asymptotic results from \cite{ffformula} and \cite{anisoformula} together with  (\ref{mix}). Thus we obtain
\begin{eqnarray}
&&u^{\infty}( \hat{x},\hat y)= \gamma \eps^d k^2 \sum\limits_{m=1}^{M }   |B_m| \big(n_m-n  \big) u_b(z_m,- {\hat x})u_b(z_m,\hat y)  \nonumber \\
&& \hspace{1in} +  \gamma \eps^d \sum\limits_{m=1}^{M }   {\bf M}^{(m)} \grad  u_b(z_m,- {\hat x}) \cdot \grad u_b(z_m,\hat y) + {o} (\eps^{d}), \label{asymformula}
\end{eqnarray}
where the polarization tensor ${\bf M}^{(m)}$ is given by 
$$ {\bf M}^{(m)}_{i,j}=e_i \cdot (A_m-A)e_j +\int\limits_{\partial B_m} \left[ \nu(y) \cdot (A_m-A)e_j \right] \phi_i^+(y) \, ds_y$$ 
with $e_i$ being the $i$th basis vector in $\R^d$ and $\phi_i$ is the solution to 
 \begin{eqnarray*}
&&\grad \cdot {A}(x) \grad \phi_i=0 \, \,  \textrm{ in } \,  \R^d \setminus \overline{B}_m\\
&&\grad \cdot A_m \grad \phi_i=0 \, \,  \textrm{ in } \,  {B}_m\\
&&\phi_i^- - \phi_i^+=x_i   \, \,  \textrm{ on } \,  \partial B_m\\
&&\frac{\partial}{\partial \nu_{A_m}} \phi_i^- - \frac{\partial}{\partial \nu_{A}}\phi_i^+=\frac{\partial}{\partial \nu_{A}}x_i   \, \,  \textrm{ on } \,  \partial B_m\\
&& \phi_i(x)=\mathcal{O} \left(\frac{1}{|x|^{d-1}} \right)
\end{eqnarray*}

We now wish to use the leading term in \eqref{asymformula} to determine the location of the defective regions $z_m + \eps B_m$. To this end assume that there are $N$ incident and observation directions given by $\hat y_j,\hat x_i \in {\mathbb S}$ for $i,j= 1, 2, \dots N$. Now we define the multi-static response matrix ${\bf F} \in \C^{N \times N}$ given by 
\begin{eqnarray}
&&{\bf F}_{i,j}= \gamma \eps^d \sum\limits_{m=1}^{M }  k^2  |B_m| \big(n_m-n \big) u_b(z_m,- {\hat x}_i) u_b(z_m,{\hat y}_j) + \nonumber \\ 
&& \hspace{1.5in}+\gamma \eps^d \sum\limits_{m=1}^{M }   {\bf M}^{(m)} \grad  u_b(z_m,- {\hat x}_i) \cdot \grad u_b(z_m,{\hat y}_j) .  \label{msm}
\end{eqnarray}
\begin{remark}
The asymptotic expansion in \eqref{asymformula} as well as the multi-static response matrix ${\bf F}$ given by \eqref{msm} can be constructed for the more general case of an inhomogeneous background (i.e. where $A(x)$ is a matrix valued function and $n(x)$ is a scalar function in $D$). In this case $A$ and $n$ are replaced by $A(z_m)$ and $n(z_m)$.  A general mixed reciprocity for inhomogeneous media is proven in \cite{irene}.
\end{remark}
The range of multi-static response  matrix ${\bf F}$ determines the location of small inhomogeneities assuming that the data is collected at sufficiently many directions (see \cite{amari},\cite{harris-thesis}  and references therein). In particular,  we define the vectors ${\bf g}_z\in \C^N$ and $ {\bf g}_{z,b}  \in \C^N$  for any point $z \in \R^d$ and $b \neq 0 \in \C^d$ by 
$$ {\bf g}_z= \big(\, u_b(z,- {\hat x}_1) ,\,  \dots , u_b(z,- {\hat x}_N) \, \big)^{\top}$$
$${\bf g}_{z,b} = \big( b \cdot \grad u_b(z , -\hat{x}_1) , \, \dots \, , b \cdot \grad u_b(z , -\hat{x}_N) \big)^{\top},$$
and let
$${\bf g}_{z,(1,b)}= {\bf g}_z+{\bf g}_{z,b}.$$
Then the following range test can be proven  (see \cite{amari},\cite{harris-thesis}  and references therein), which essentially says that $z \in \{ z_m \, : \, m=1, \dots , M\}$ if and only if ${\bf g}_{z,(1,b)}$ is in the range of  ${\bf F}{\bf F}^*$.
\begin{theorem} \label{music7}
Let ${\bf w}_j$ be the $j$-th orthonormal eigenvector of ${\bf F}{\bf F}^*$ and let $r={\mbox{\em Rank}}\big( {\bf F}{\bf F}^* \big)$. Assume that the set $S=\{ {\hat x}_i \, : \, i \in \N \}$ is dense in $\mathbb{S}$  such that any analytic function that vanishes on $S$ also vanishes on $\mathbb{S}$. If $z \in D$ then there is a number $N_0 \in \N$ such that for all $N \geq N_0$ we have that 
$$\mathcal{I}(z)= \left[ \sum\limits_{j=r+1}^N \left|\big( {\bf g}_{z,(1,b)} , {\bf w}_j \big) \right|_{\ell^2}^2 \right]^{-1} < \infty \quad  \text{ if and only if } \quad z \in \{ z_m \, : \, m=1, \dots , M\}.$$
\end{theorem}
\subsection*{Numerical Validation of the MUSIC Algorithm}
We present here the numerical implementation of the MUSIC algorithm in the ${\mathbb R}^2$ case. To this end, we use simulated far-field data to reconstruct the defects in a square scatterer. The simulated data comes from solving the direct scattering problems \eqref{musicprod1}-\eqref{musicprod2} and \eqref{musicprod3}-\eqref{musicprod4} using a cubic finite element method with a perfectly matched layer. From this we will have the approximated scattered fields $u^{s}_\eps (\cdot \, , \hat x)$ and $u^{s}_b( \cdot \, , \hat x)$. The multi-static response matrix will be defined as 
$${\bf F}=\big[ u^{\infty}_\eps (\hat{x}_i,\hat{x}_j) - u^{\infty}_b (\hat{x}_i,\hat{x}_j) \big]_{i,j=1}^{N},$$
where the far-field patterns are given by the solutions of the direct problems using the finite element method. In the following we use $N$ different directions on the unit circle given by 
$$\hat{x}_i=\Big(\cos \left(  { 2 \pi (i-1 )}/{N}\right), \sin \left(  { 2 \pi (i-1 )}/{N}\right) \Big) \quad \text{ for } \quad i=1, \dots , N.$$ 
 In all examples we take $D=[-2,2]^2$ and we fix the wave number $k=1$. 
 
 We want to illustrate the performance of the MUSIC algorithm in reconstructing the defective regions $z_m + \eps B_m$ inside $D$. 
 We give examples with random noise added to the simulated data for $u^{\infty}_\eps (\hat{x}_i,\hat{x}_j)$. The random noise level is given by $\delta$ where the noise is added to the far-field data $u^{\infty}_\eps (\hat{x}_i,\hat{x}_j) +\delta E_{i,j}$ and the random matrix $E$ is such that $\| E \|_2 =1$. Since we have that $A$, $\, n$ and $D$ are known for non-destructive testing we can assume that the far-field pattern $u^{\infty}_b (\hat{x}_i,\hat{x}_j)$ is computed  from solving the scattering problem for the known background. Reconstruction examples are presented in Figure \ref{musicfig1}, Figure \ref{musicfig3}  and Figure \ref{musicfiganiso} where the configuration and reconstruction parameters are explained in the respective  labels.
 \begin{figure}[!ht]
\centering
\includegraphics[scale=0.3]{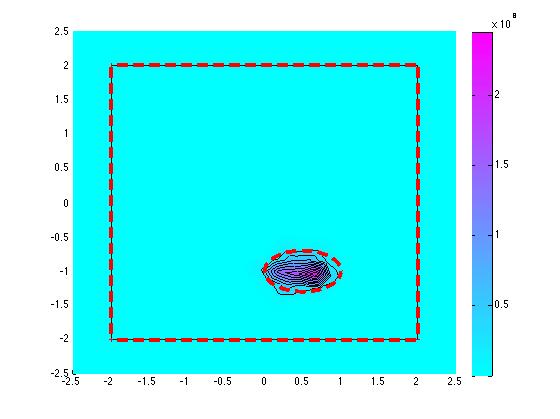}\includegraphics[scale=0.3]{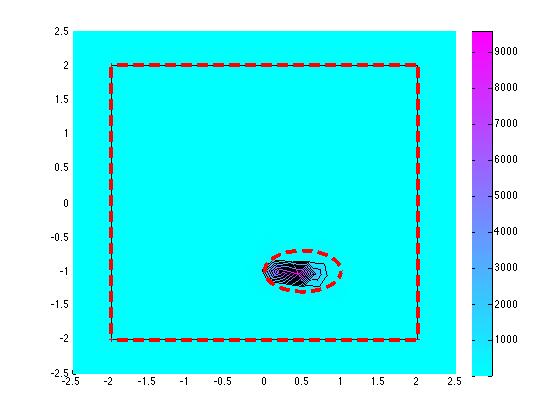}\\
\caption{Reconstruction the an ellipse $D_\epsilon$  centered at $(0.5,-1)$ with axes equal $0.5$ and $0.3$ inside $D:=[-2,2]\times[-2,2]$. The material parameters in $D$  are $A=0.5I$ and $n=5$, and in $D_\epsilon$ are $A_1=I$ and $n_1=1$, i.e. the defective region is a void.  The figure on the left shows the reconstruction without noise and on the right  with 10$\%$ noise. Here $N=64$.} 
\label{musicfig1}
\end{figure}
\begin{figure}[!ht]
\centering
\includegraphics[scale=0.3]{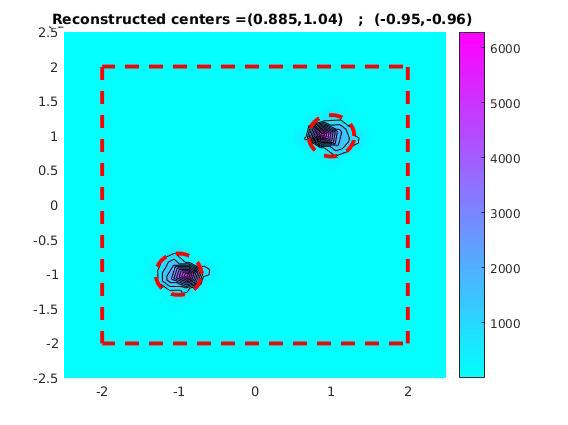}\includegraphics[scale=0.3]{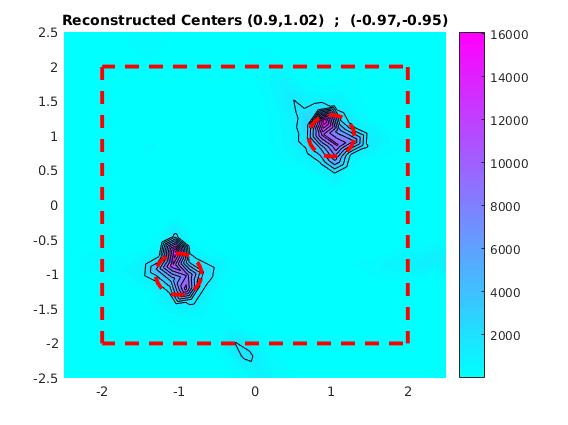}\\
\caption{Reconstruction  the defective region $D_\epsilon$ which is the union of discs  centered  centered at $(1,1)$ and $(-1,-1)$ with radius $\eps=0.3$ inside  $D:=[-2,2]\times[-2,2]$. The material parameters in $D$ are $A=0.5I$ and $n=5$, and in $D_\epsilon$  are $A_1=I$ and $n_1=1$.  The figure on the left shows the reconstruction without noise and on the right with 10$\%$ noise. Here $N=20$. Notice the estimated centers of the reconstructed discs.}
\label{musicfig3}
\end{figure}

\begin{figure}[!ht]
\centering
\includegraphics[scale=0.3]{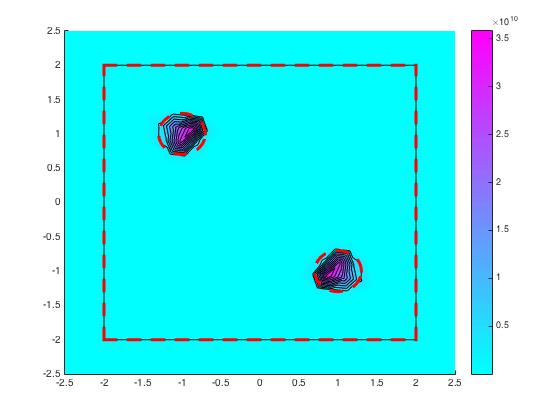}\includegraphics[scale=0.3]{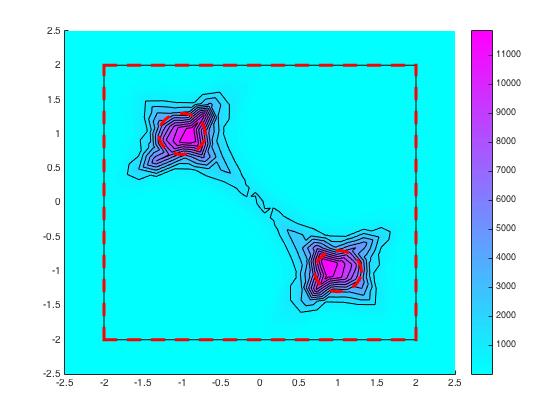}\\
\caption{Reconstruction the defective region $D_\epsilon$ which is the union of discs  centered  centered at $(-1,1)$ and $(1,-1)$ with radius $\eps=0.3$ inside  $D:=[-2,2]\times[-2,2]$. The material parameters in the anisotropic $D$ are $A=[  10 \, \, \,  1\,  ; \, 1 \, \,  \, 10 ]$ and  $n=5$,  and in $D_\epsilon$  are $A_1=I$ and $n_1=1$. The figure on the left shows the reconstruction without noise and on the right with $2\%$ noise. Here $N=32$.}
\label{musicfiganiso}
\end{figure}
\section{Convergence of the Transmission Eigenvalue Spectrum}
Recall that the multi-static data used for the MUSIC algorithm can also determine the transmission eigenvalues corresponding to the perturbed media. Having reconstructed the location of the small defects, we would like to obtain information about the strength of the perturbations $A_m$ and $n_m$ from the transmission eigenvalues.  To this end, we investigate how the small defects affect the transmission eigenvalues.  For the convergence analysis of the transmission eigenvalues  we assume more regularity on the coefficients of the unperturbed (without defects) media, i.e. they are  given by the symmetric matrix $A(x) \in {C}^2(D, \R^{d \times d})$ and $n(x) \in {C}^1(D)$. We start by showing that the eigenvalues $k_\eps$ for the perturbed media converge to the eigenvalues for the unperturbed media as $\eps \rightarrow 0$. Then we derive an asymptotic formula with correction term of the first order that can be used to obtain more information about the small defects.  To analyze \eqref{tedefect1}-\eqref{tedefect2} we define the variational space  $X(D):=\Big\{ (w,v): \; w,v \in H^1(D) \, | \,  w-v \in H^1_0(D) \Big\}$
equipped with the $H^1(D)\times H^1(D)$ inner product. It is clear that the variational form of \eqref{tedefect1}-\eqref{tedefect2} is given by 
\begin{equation*}
\int\limits_D A_\eps \grad w \cdot \grad \overline{\varphi}_1 -\grad v \cdot \grad \overline{\varphi}_2 -k_\eps^2(n_\eps w \overline{\varphi}_1- v \overline{\varphi}_2 )\, dx =0 \qquad \text{for all} \, \, (\varphi_1,\varphi_2) \in X(D). 
\end{equation*}
 For convenience we define the  bounded sesquilinear forms
$$\mathcal{A}_\eps \big( (w,v) ; (\varphi_1,\varphi_2) \big):=\int\limits_{D} A_\eps \grad w \cdot \grad \overline{\varphi}_1 + A_{min} w \overline{\varphi}_1 \, dx -\int\limits_{D}   \grad v \cdot \grad \overline{\varphi}_2 + v \overline{\varphi}_2 \, dx,$$
$$\mathcal{B}_\eps \big( (w,v) ; (\varphi_1,\varphi_2) \big):=\int\limits_{D}  n_\eps w \overline{\varphi}_1 \, - v \overline{\varphi}_2 \, dx,$$
$$\mathcal{C} \big( (w,v) ; (\varphi_1,\varphi_2) \big):=\int\limits_{D} A_{min} w \overline{\varphi}_1 \,  - v \overline{\varphi}_2 \, dx.$$
Therefore we have that \eqref{tedefect1}-\eqref{tedefect2} can be written as for all $(\varphi_1,\varphi_2) \in X(D)$
\begin{equation} 
\mathcal{A}_\eps \big( (w,v) ; (\varphi_1,\varphi_2) \big)- k_\eps^2 \mathcal{B}_\eps \big( (w,v) ; (\varphi_1,\varphi_2) \big) -\mathcal{C} \big( (w,v) ; (\varphi_1,\varphi_2) \big)=0. \label{tedefectvarform}
\end{equation}
Let us define by ${\bf A}_{\epsilon},\, {\bf B}_{\epsilon}$ and ${\bf C} :X(D)\to X(D)$ the bounded linear operators defined from $\mathcal{A}_\eps \big( \cdot \, ;  \cdot \big)$, $\mathcal{B}_\eps \big( \cdot \, ;  \cdot \big)$ and $\mathcal{C} \big( \cdot \, ;  \cdot \big)$ by  means of the Riesz representation theorem. The unperturbed media corresponds to  $\eps =0$, where $A_0:=A$, $n_0:=n$ and $k_0:=k$. It can be shown using $\mathbb{T}$-coercivity that  if $A_{min}$ and $a_{min} >1$ that ${\bf A}_{\epsilon}$ is invertible with the norm of the inverse independent of $\eps \geq 0$. To this end we consider the isomorphism $\mathbb{T}(w,v)=(w, -v+2w): X(D) \mapsto X(D)$ (it is easy to check that $\mathbb{T}=\mathbb{T}^{-1}$).
Then 
\begin{equation*}
\hspace*{-9cm} \left| \mathcal{A}_\eps \big( (w,v) ; \mathbb{T}(w,v) \big) \right| \geq 
\end{equation*}
$$\int\limits_{D} A_\eps \grad w \cdot \grad \overline{w} + A_{min} |w|^2 \, dx +\int\limits_{D}  | \grad v|^2 +|v|^2 \, dx-2\left| \,  \int\limits_{D}  \grad v_\eps \cdot \grad \overline{w}_\eps +  v_\eps \overline{w}_\eps \, dx \right|$$
and by Young's inequality we obtain that 
\begin{eqnarray*}
\left| \mathcal{A}_\eps \big( (w,v) ; \mathbb{T}(w,v) \big) \right| &\geq& \left(\alpha- \frac{1}{\delta} \right) ||w_\eps||^2_{H^1(D)} +(1-\delta) ||v_\eps||^2_{H^1(D)}.
\end{eqnarray*}
where we let $\alpha = \min \{ A_{min}, \, a_{min} \}$. Therefore we have proven that $\mathcal{A}_\eps \big( (w,v) ; \mathbb{T} (w,v) \big)$ is coercive provided that $\delta \in (1/\alpha , \, 1)$, implying ${\bf A}_{\epsilon}$ is invertible for $\epsilon\geq 0$.  Similar arguments hold for $A_{max}$ and $a_{max} <1$ where $A_{min}$ is replaced by $A_{max}$ in $\mathcal{A}_\eps \big( \cdot \, ;  \cdot \big)$ and $\mathcal{C} \big( \cdot \, ;  \cdot \big)$ with $\mathbb{T}(w,v)=(w-2v, -v)$. It is clear that in either case that ${\bf B}_{\epsilon}$ and ${\bf C}$ are compact operators by appealing to the compact embedding of $H^1(D)$ in $L^2(D)$. Now by \eqref{tedefectvarform} it is clear that $(w,v)$ are eigenfunctions corresponding to the eigenvalue $k_\eps$ provided that 
\begin{equation} 
\big( {\bf I} -k_\eps^2{\bf A}_{\eps}^{-1} {\bf B}_{\eps}-{\bf A}_{\eps}^{-1}{\bf C} \big)(w,v)=(0,0).
\end{equation}
Let us denote the eigenvalue parameter $\tau_\eps=k^2_\eps$ and define ${\bf T}_\eps:X(D)\to X(D)$
\begin{equation} 
{\bf T}_\eps (\tau_\eps):={\bf A}_{\eps}^{-1} {\bf B}_{\eps}+ \frac{1}{\tau_\eps} {\bf A}_{\eps}^{-1}{\bf C}.  \label{tdef}
\end{equation}
We can now rephrase \eqref{tedefectvarform} as a non-linear eigenvalue problem 
 \begin{equation} 
\tau_\eps {\bf T}_\eps (\tau_\eps)(w,v)=(w,v), \qquad \qquad \epsilon\geq 0. \label{nonlinear}
\end{equation}
Note that it is clear that ${\bf T}_\eps (\tau)$ depends analytically on $\tau$ in any subset of the complex plane that does not include the origin.
\subsection*{Convergence of the Spectrum}

In this section, we study the convergence of ${\bf T}_\eps (\tau)$ in the operator norm to the unperturbed operator ${\bf T}_0 (\tau)$ and then use results from \cite{sharinonlinear} to prove convergence for the transmission eigenvalues and eigenfunctions. To this end, notice that ${\bf B}_{\eps}$ and ${\bf C}$ are compact operators and that $\| {\bf A}_{\eps}^{-1} \|$ is uniformly bounded with respect to $\epsilon$, so we can conclude that ${\bf T}_\eps (\tau)$ is compact for all $\eps \geq 0$. Hence the convergence of ${\bf T}_\eps (\tau)$ would then imply the convergence of the transmission eigenvalues. We start by studying the convergence of the operator ${\bf B}_{\eps}$ to ${\bf B}_{0}$.
\begin{theorem} \label{bconverge}
${\bf B}_{\epsilon} \rightarrow {\bf B}_{0}$ in the operator norm. Moreover for some $\alpha \in (0,1)$ we have that $\left\| {\bf B}_{\eps}- {\bf B}_{0} \right\| \leq C \eps^{\alpha} \, \text{ in } \, \R^d$ for some C independent of $\eps$, $d=2,3$. 
\end{theorem}

\begin{proof}
By definition we have that 
\begin{eqnarray*}
 \left| \big( {\bf B}_{\eps}(w,v) ; (\varphi_1,\varphi_2) \big) -\big( {\bf B}_{0}(w,v) ; (\varphi_1,\varphi_2) \big) \right| &=&  \left| \, \int\limits_{D_\eps}  (n_\eps- n) w \overline{\varphi}_1 \, dx \right|   \\
 										& \leq & || (n_\eps - n) w||_{L^2(D_\eps)} || (\varphi_1,\varphi_2) ||_{X(D)}.
 \end{eqnarray*}
 
 \noindent Therefore, we have that  $\left\| \big( {\bf B}_{\eps}-{\bf B}_{0}\big) (w,v) \right\|_{X(D)} \leq  || (n_\eps - n) w||_{L^2(D_\eps)}$. Now since $w \in H^1(D)$ we have from Sobolev's embedding in $\R^2$ or $\R^3$ that $w\in L^p(D)$ for some $p \geq 2$ (see e.g. \cite{brezis} for embedding results). We then conclude that $|w|^2 \in L^{p/2}(D)$. Now let $q$ be  defined by $\frac{1}{p/2}+\frac{1}{q}=1$ notice that $\frac{1}{q}=\frac{p-2}{p}$. Therefore by using the duality between $L^{p/2}(D)$ and $L^{q}(D)$ along with Sobolev's  embedding we have that 
\begin{eqnarray*}
\left\| \big( {\bf B}_{\eps}-{\bf B}_{0}\big) (w,v) \right\|^2_{X(D)} &\leq&  ||(n_\eps - n)||^2_{\infty} ||w||^2_{L^2(D_\eps)}\leq  C  ||\,  |w|^2 ||_{L^{p/2}(D)} || \chi_{D_\eps}||_{L^{q}(D)}\\
										  &=& C  |D_\eps| ^{1/q} || w ||^2_{L^{p}(D)} \leq  C \eps^{d/q}  || (w,v) ||^2_{X(D)}.
 \end{eqnarray*}
Hence, we have that  
$$||  {\bf B}_{\eps}-{\bf B}_{0}|| \leq C \eps^{d/2q} \quad \text{for } \, d=2,3$$
where the constant $C$ incorporates the norm of the contrasts but is independent of $\eps$. Now for the $\R^2$ for any choice of $p > 2$ we have that $\frac{1}{q} <1$ giving the result. For the case in $\R^3$ we can choose $p < 6$ giving that $\frac{1}{q}<2/3$ and therefore $d/2q <1$, which gives the result in $\R^3$.
 \end{proof}

We are now interested in the convergence of ${\bf A}_{\eps}^{-1} {\bf B}_{\eps}$ and ${\bf A}_{\eps}^{-1} {\bf C}$ as $\eps$ tends to zero.  Recall that ${\bf A}_{\eps}^{-1}$ exists as a bounded linear operator for all $\eps \geq 0$ where the norm of ${\bf A}_{\eps}^{-1}$ is uniformly bounded with respect to $\eps$. 
To study the convergence of ${\bf A}_{\eps}^{-1} {\bf B}_{\eps}$ and ${\bf A}_{\eps}^{-1} {\bf C}$  we first need some regularity results pertaining to ${\bf B}_0$ and ${\bf C}$. Notice that by the variational definition of ${\bf B}_{0}$ we have that for any $(f,g)\in X(D)$ if we denote ${\bf B}_0 (f,g)=(w,v)$ then 
\begin{equation} 
-\Delta w+w=n f \quad \text{ and } \quad \Delta v+v=g \quad \text{in } \, D. \label{pdeforb}
\end{equation}

\noindent Therefore by elliptic regularity we have that $w$ and $v$ are in $H^3_{loc}(D)$ provided that $n$ is continuously differentiable, and for any $\Omega \subset D$ 
\begin{equation*} 
\| w \|_{H^3(\Omega)} +\| v \|_{H^3(\Omega)} \leq C \left( \| f \|_{H^1(D)} +\| g \|_{H^1(D)} \right).
\end{equation*}

Next, as for the operator ${\bf C}$ we have that for any $(f,g)\in X(D)$ if we denote ${\bf C}(f,g)=(w,v)$ then 
\begin{equation*} 
-\Delta w+w=A_{min}f \quad \text{ and } \quad \Delta v+v=g \quad \text{in } \, D, 
\end{equation*}
and we have the elliptic regularity estimates 
for any $\Omega \subset D$ 
\begin{equation*} 
\| w \|_{H^3(\Omega)} +\| v \|_{H^3(\Omega)} \leq C \left( \| f \|_{H^1(D)} +\| g \|_{H^1(D)} \right).
\end{equation*}

\begin{theorem} \label{aconverge}
We have that 
$${\bf A}_{\eps}^{-1} {\bf B}_{\eps} \rightarrow {\bf A}_{0}^{-1} {\bf B}_{0} \quad \text{ and } \quad {\bf A}_{\eps}^{-1} {\bf C} \rightarrow {\bf A}_{0}^{-1} {\bf C}$$
in the operator norm as $\eps \rightarrow 0$. 
\end{theorem}
\begin{proof}
Consider the pair $(w_\eps , v_\eps)$ and $(w,v)$ in $X(D)$ defined by  $(w_\eps , v_\eps) = {\bf A}_{\eps}^{-1}(f,g)$ and  $(w , v) = {\bf A}_{0}^{-1}(f,g)$  for any $(f,g)\in X(D)$. By definition we have that 
\begin{eqnarray*}
\mathcal{A}_\eps  \big( (w-w_\eps ,v-v_\eps) ; (\varphi_1,\varphi_2) \big)= \int\limits_{D_\eps} (A_\eps-A)\grad w \cdot \grad \overline{\varphi}_1 \, dx,
 \end{eqnarray*}
whence using the $\mathbb{T}$-coercivity  we conclude that 
$$\left\| \big({\bf A}_{\eps}^{-1}-{\bf A}_{0}^{-1} \big) (f,g) \right\|_{X(D)} \leq C || (A_\eps-A) \grad w ||_{L^2(D_\eps)}.$$
Next we  have that ${\bf A}_{0}^{-1}{\bf B}_0(f,g)=(w,v)$ due to the variational form of ${\bf A}_{0}$ satisfies
$$ -\grad \cdot A(x) \grad w+A_{min} w=-\Delta p+p \quad \text{ and } \quad \Delta v+v=-\Delta q+q \quad \text{in } \, D.$$
Recalling ${\bf B}_0 (f,g)=(p,q)$ we have $\| p \|_{H^3(\Omega)} +\| q \|_{H^3(\Omega)} \leq C \left( \| f \|_{H^1(D)} +\| g \|_{H^1(D)} \right)$ where $\Omega \subset D$ and  by elliptic regularity given any $\Omega' \subset \Omega  \subset D$ we have that 
$$\| w \|_{H^3(\Omega' )} +\| v \|_{H^3(\Omega'  )} \leq C \left( \| p \|_{H^3(\Omega)} +\| q \|_{H^3(\Omega)} \right) \leq C || (f,g) ||_{X(D)}.$$
Fixing $\Omega'$ and $\Omega$ such that $D_\eps \subset \Omega' \subset \Omega \subset D$ for all $\eps$ sufficiently small and using that $H^3(\Omega') \subset C^1(\Omega')$ we have the following estimates 
 \begin{eqnarray*}
\left\| \big({\bf A}_{\eps}^{-1}-{\bf A}_{0}^{-1} \big) {\bf B}_0(f,g) \right\|_{X(D)}  & \leq& C  || w ||_{C^1(\Omega')} || \chi_{D_\eps} ||_{L^2(D)}\\
												& \leq & C \eps^{d/2} || w ||_{C^1(\Omega')}.
 \end{eqnarray*}

Now appealing to the continuity of the embedding of $H^3(\Omega')$ into $C^1(\Omega')$ and the regularity estimate we have that 
$$\left\| \big({\bf A}_{\eps}^{-1}-{\bf A}_{0}^{-1} \big) {\bf B}_0 (f,g) \right\|_{X(D)} \leq C  \eps^{d/2} || (f,g) ||_{X(D)}.$$
Using that 
$$ \left\| {\bf A}_{\eps}^{-1}{\bf B}_\eps-{\bf A}_{0}^{-1}{\bf B}_0 \right\| \leq   \left\| {\bf A}_{\eps}^{-1} \big( {\bf B}_{\eps}-{\bf B}_{0}\big)  \right\|+ \left\| \big({\bf A}_{\eps}^{-1}-{\bf A}_{0}^{-1} \big) {\bf B}_0 \right\|$$
along with the uniform boundedness of $|| {\bf A}_{\eps}^{-1} ||$ and the norm convergence of ${\bf B}_\eps$ to ${\bf B}_0$ implies that ${\bf A}_{\eps}^{-1} {\bf B}_{\eps} \rightarrow {\bf A}_{0}^{-1} {\bf B}_{0} $ in norm. The same arguments work for showing that ${\bf A}_{\eps}^{-1} {\bf C} \rightarrow {\bf A}_{0}^{-1} {\bf C}$ in norm, which ends the proof.
\end{proof}
 
\begin{corollary}\label{convcorollary}
Let the operators ${\bf A}_{\eps}$, ${\bf A}_{0}$, ${\bf B}_{\eps}$, ${\bf B}_{0}$ and ${\bf C}$ be defined by the variational forms given above. Then we have that for $d=2,3$ 
\begin{equation*} 
\left\| \big({\bf A}_{\eps}^{-1}-{\bf A}_{0}^{-1} \big) {\bf B}_0 \right\|=\mathcal{O}(\eps^{d/2}), \qquad  \left\|\big({\bf A}_{\eps}^{-1}-{\bf A}_{0}^{-1} \big) {\bf C} \right\|=\mathcal{O}(\eps^{d/2}),
\end{equation*}
and $\left\|{\bf A}_{\eps}^{-1}\big({\bf B}_{\eps}-{\bf B}_{0} \big)  \right\|=\mathcal{O}(\eps^{\alpha})$ for some $\alpha \in (0,1)$ .
\end{corollary} 
Combining the above results we have:
\begin{theorem} \label{tconverge}
Let the operator ${\bf T}_{\eps}(\tau)$ be as defined in \eqref{tdef} and $\tau \in  U$ with $U$ being any bounded subset of $\C$ with zero not a limit point of $U$. Then we have that 
\begin{equation*} 
\left\| {\bf T}_{\eps}(\tau)-{\bf T}_{0}(\tau) \right\|\longrightarrow 0  \quad  \text{ as } \quad  \eps \rightarrow 0.
\end{equation*} 
Moreover if $n_\eps=n$ for all $\eps\geq 0$ then we have that 
$$ \left\| {\bf T}_{\eps}(\tau)-{\bf T}_{0}(\tau) \right\| =\mathcal{O}(\eps^{d/2}). $$
\end{theorem}
Having proven the convergence of the operator ${\bf T}_{\eps}(\tau)$ we are ready to study the convergence of the real transmission eigenvalues using the abstract result  from \cite{sharinonlinear}.
\begin{lemma}\label{sharilemma}
Let $\tau$ be a non-linear eigenvalue of ${\bf T}_0$ and assume that ${\bf T}_0(\tau)$ and ${\bf T}_\eps (\tau)$ are both meromorphic in some region $U$ of $\C$ containing $\tau$. Also assume that ${\bf T}_\eps (\tau) \rightarrow {\bf T}_0(\tau)$ in the operator norm. Then for any ball around $\tau$ there exists a $\eps_0 > 0$ such that ${\bf T}_\eps$ has a non-linear eigenvalue in the ball for all $\eps < \eps_0$. Conversely if $\tau_\eps$ is a sequence of non-linear eigenvalues of ${\bf T}_\eps$ that converges as $\eps \rightarrow 0$, then the limit $\tau$ is a non-linear eigenvalue of ${\bf T}_0$. 
\end{lemma}

By Theorem \ref{tconverge} we have that ${\bf T}_\eps (\tau) \rightarrow {\bf T}_0(\tau)$ in the operator norm in any in region $U$ of $\C \setminus \{0\}$ and from the definition of the operator ${\bf T}_\eps (\tau)$ we have that it depends analytically on $\tau$ in any subset of the complex plane that does not include the origin. Finally to conclude the convergence of the eigenvalues, we need bounds on the eigenvalues independent of $\epsilon$.  The existence of real transmission eigenvalues  and monotonicity property with respect to the refractive index  are proven in \cite{cakginhad} and  \cite{cakonikirsch}. The monotonicity property implies $\epsilon$-independent bounds  on these real transmission eigenvalues since $A_\epsilon$ and $n_\epsilon$ are bounded above and below uniformly with respect to $\epsilon>0$ (more specifically such bounds can be obtained by modifying the proof of Theorem 2.6 and Theorem 2.10 in \cite{cakonikirsch} in a similar way as in the proof of Corollary 2.6 in  \cite{cakginhad}.)

\section{Asymptotic Formula for the Transmission Eigenvalues }  
Having proven the convergence of the transmission eigenvalues, we now want to obtain an asymptotic formula for the real transmission eigenvalues. To this end, we need to construct an appropriate corrector that will give an explicit formula for the first term in the asymptotic expansion for the transmission eigenvalues. For technical reasons that have to do with the rate of convergence of  ${\bf B}_\epsilon$ to ${\bf B}_0$ (which will be explained later on) we derive this corrector for the case when there is no contrast in the lower term, i.e.  $n_\epsilon=n$. To avoid technicalities in the presentation, the corrector will be derived for a homogeneous anisotropic media and the results can be generalized for an inhomogeneous media as in \cite{CMR}. Hence in this section we again assume that the coefficients $A$ and $n$ are constant in $D$.

\subsection*{Correction for the Operator ${\bf A}_{\eps}^{-1} - {\bf A}_{0}^{-1}$}
Consider the pair $(w_\eps , v_\eps)$ and $(w,v)$ in $X(D)$ defined by 
 \begin{eqnarray}
(w_\eps , v_\eps) = {\bf A}_{\eps}^{-1}(f,g) \quad \text{ and } \quad (w , v) = {\bf A}_{0}^{-1}(f,g) \label{eqwv} 
 \end{eqnarray}
and we assume that $w$ is a smooth function.   Without loss of generality in the following we perform the calculations only for one inhomogeneity. For multiple inhomogeneity one simply sum the correctors. To this end, assume that the defective region is of the form $\eps B$ where $B$ is centered at the origin with constant matrix $A_1$ being the constitutive parameter. We make the scaling $y=x/\eps$ and $\widetilde D=\frac{1}{\eps} D$ and let $w_\eps^{(1)} (y) \in H^1_0(\widetilde D)$ be the unique solution to 
 \begin{eqnarray}
\int\limits_{\widetilde D} \widetilde A \grad_y w_\eps^{(1)} \cdot \grad_y \overline{\varphi} + A_{min} w_\eps^{(1)} \overline{\varphi} \, \, dy= \int\limits_{\partial B} \big[ (A_1-A)\grad_x w(0) \cdot \nu \big] \overline{\varphi} \, \, ds_y \label{c1prob}
  \end{eqnarray}
with $\widetilde A=A_1 \, \chi_B + A(1-\chi_B)$.  

\begin{theorem}\label{errorthm}
Assume that  $(w_\eps , v_\eps)$ and $(w,v)$ are defined by \eqref{eqwv} with $w$ being a smooth function, then we have that 
 \begin{eqnarray}
\| w_\eps(x)-w(x)-\eps w(0)w^{(1)}_\eps(x/\eps) \|_{H^1(D)} + \| v_\eps(x)-v(x) \|_{H^1(D)}  =\mathcal{O}(\eps^{d/2+1}). \label{error1}
 \end{eqnarray}
\end{theorem}
\begin{proof}
Recall that $x=\eps y$ and we define the error functions in $X(\widetilde D)$ $\big($note that $w_\eps^{(1)} (y) \in H^1_0(\widetilde D)$$\big)$
$$ e^w_\eps = w_\eps(\eps y)-w(\eps y)-\eps w(0) w^{(1)}_\eps(y) \quad \text{and} \quad e^v_\eps = v_\eps(x)-v(x). $$
Now let $(\varphi_1 , \varphi_2) \in {X(\tilde D) } $ and define the sesquilinear form 
\begin{eqnarray*}
\widetilde{\mathcal{A}}_\eps \big( (e^w_\eps ,e^v_\eps ) ; (\varphi_1,\varphi_2) \big) := \int\limits_{\widetilde D} \widetilde A \grad_y e^w_\eps \cdot \grad_y \overline{\varphi}_1 + A_{min} e^w_\eps \overline{\varphi}_1 \, dy -\int\limits_{\widetilde D}   \grad_y e^v_\eps \cdot \grad_y \overline{\varphi}_2 + e^v_\eps \overline{\varphi}_2 \, dx.
 \end{eqnarray*}
Using \eqref{eqwv} we have that
\begin{eqnarray*}
 &&\widetilde{\mathcal{A}}_\eps \big( (e^w_\eps ,e^v_\eps ) ; (\varphi_1,\varphi_2) \big) = \int\limits_{B} (A_1-A) \grad_y w(\eps y) \cdot \grad_y \overline{\varphi}_1  \, dy \\
&& \hspace{2in} - \eps w(0) \int\limits_{\widetilde D} \widetilde A \grad_y w_\eps^{(1)} \cdot \grad_y \overline{\varphi}_1 + A_{min} w_\eps^{(1)} \overline{\varphi}_1 \, \, dy.
 \end{eqnarray*}
Using integration by parts and \eqref{c1prob} gives that  
\begin{eqnarray*}
&&\widetilde{\mathcal{A}}_\eps \big( (e^w_\eps ,e^v_\eps ) ; (\varphi_1,\varphi_2) \big) = \eps^2  \int\limits_{B} \overline{\varphi}_1 \grad_x \cdot (A-A_1) \grad_x w(\eps y)    \, dy \\
 && \hspace{1in} + \eps w(0)\int\limits_{\partial B} \Big[ (A_1-A)( \grad_x w(\eps y)-\grad_x w(0)) \cdot \nu \Big] \overline{\varphi}_1 \, \, ds_y.
  \end{eqnarray*}
  
Recall that $w$ is smooth, therefore $ \grad_x \cdot (A-A_1) \grad_x w(\eps y)$ is bounded in $B$. Also notice that by Taylor's expansion we have that the term $( \grad_x w(\eps y)-\grad_x w(0)) = \mathcal{O}(\eps)$. Therefore, we can conclude that there is a constant $C$ independent of $\eps$ such that  
$$ \left| \widetilde{\mathcal{A}}_\eps \big( (e^w_\eps ,e^v_\eps ) ; (\varphi_1,\varphi_2) \big)  \right| \leq C \eps^2 \| (\varphi_1 , \varphi_2) \|_{H^1(\tilde D) \times H^1(\tilde D) } $$
Using the $\mathbb{T}$-coercivity of the sesquilinear form $\tilde{\mathcal{A}}_\eps \big( \cdot \, ; \cdot \big)$ in $X(\widetilde D)$ gives that 
\begin{eqnarray}
\| w_\eps(\eps y)-w(\eps y)-\eps w(0) w^{(1)}_\eps(y) \|_{H^1(\widetilde D)} + \| v_\eps(\eps y)-v(\eps y) \|_{H^1( \widetilde D)}  \leq C \eps^2 , \label{unscaled}
  \end{eqnarray}
and the result follows from scaling. 
\end{proof}

Notice that from \eqref{c1prob} we have that $\| w^{(1)}_\eps(y) \|_{H^1(\widetilde D)}$ is bounded independently of $\eps$ by the Lax-Milgram lemma. Therefore by scaling we have that $\| w^{(1)}_\eps(x/\eps) \|_{H^1(D)} \leq C \eps^{d/2-1}$ with $C$ independent of $\eps$, which gives the following result. 

\begin{corollary}
Assume that  $(w_\eps , v_\eps)$ and $(w,v)$ are defined by \eqref{eqwv} with $w$ being a smooth function then we have that 
 \begin{eqnarray}
\| w_\eps(x)-w(x) \|_{H^1(D)} +\| v_\eps(x)-v(x) \|_{H^1(D)} =\mathcal{O}(\eps^{d/2}). \label{error2}
 \end{eqnarray}
\end{corollary}

Notice that the corrector $w^{(1)}_\eps(y)$ depends on $\eps$, hence we now wish to construct a corrector that is independent of the small parameter $\eps$. To this end, we define the function $w^{(1)} (y) \in H^1(\R^d)$ such that for all $\varphi \in H^1(\R^d)$
 \begin{eqnarray}
\int\limits_{\R^d} \widetilde A \grad_y w^{(1)} \cdot \grad_y \overline{\varphi} + A_{min} w^{(1)} \overline{\varphi} \, \, dy= \int\limits_{\partial B} \big[ (A_1-A)\grad_x w(0) \cdot \nu \big] \overline{\varphi} \, \, ds_y \label{c2prob}
  \end{eqnarray}
Note that the variational problem \eqref{c2prob} implies that 
$$-   \grad_y \cdot \widetilde A \grad_y w^{(1)} + A_{min} w^{(1)}=0 \quad \text{in } \quad \R^d \setminus \partial B.$$ 
We now have $\left|w^{(1)} \right| \rightarrow 0$ as $|y| \rightarrow \infty$, exponentially fast \cite{coltonkress}. This gives that $\grad_y w^{(1)}$ decays faster than the gradient of a solution to Laplace's equation, therefore 
$$ \|  \grad_y w^{(1)}(x/\eps) \|_{L^{\infty}(\partial D)} = o(\eps^d) \quad \text{ for } \quad d=2,3.$$

\begin{theorem}
Let $ w^{(1)}_\eps$ and $w^{(1)}$ be defined as the solutions to \eqref{c1prob} and \eqref{c2prob} respectively, then we have that 
 \begin{eqnarray*}
&&\| w^{(1)}_\eps(x/ \eps) -w^{(1)} (x/\eps)  \|_{H^1(D)}  =o(\eps^{d/2+2}) \,\,\, \text{ for } \,\, d=2, \\
&&\| w^{(1)}_\eps(x/ \eps) -w^{(1)} (x/\eps)  \|_{H^1(D)}  =o(\eps^{d/2+5/2}) \,\,\, \text{ for } \,\, d=3.
 \end{eqnarray*}
\end{theorem}  
\begin{proof}
Let $u_\eps =w^{(1)}_\eps(y) -w^{(1)} (y)$, there exists a constant $\alpha >0$ such that 
 \begin{eqnarray*}
&& \hspace{-0.2in} \alpha \| u_\eps  \|_{H^1(\widetilde D)} ^2 \leq  \int\limits_{\widetilde D} \tilde A \grad_y u_\eps \cdot \grad_y \overline{u}_\eps + A_{min} |u_\eps|^{2} \, \, dy\\
&& \hspace{0.2in} =  \int\limits_{\widetilde D} \widetilde A \grad_y w^{(1)}_\eps \cdot \grad_y \overline{u}_\eps + A_{min} w^{(1)}_\eps \overline{u}_\eps \, \, dy  -  \int\limits_{\widetilde D} \widetilde A \grad_y w^{(1)} \cdot \grad_y \overline{u}_\eps + A_{min} w^{(1)} \overline{u}_\eps \, \, dy \\
&& \hspace{0.2in} =  \int\limits_{\partial B} \big[ (A_1-A)\grad_x w(0) \cdot \nu \big] \overline{u}_\eps \, \, ds_y  -  \int\limits_{\widetilde D} \widetilde A \grad_y w^{(1)} \cdot \grad_y \overline{u}_\eps + A_{min} w^{(1)} \overline{u}_\eps \, \, dy 
 \end{eqnarray*}
Notice that the variational form \eqref{c2prob} implies that 
$$\left( A \frac{\partial  w^{(1)} }{\partial \nu_y} \right)^+ - \left( A_1 \frac{\partial  w^{(1)} }{\partial \nu_y} \right)^-  = (A_1-A)\grad_x w(0) \cdot \nu  \quad \text{ on } \quad \partial B.$$
Therefore integration by parts gives that  
 \begin{eqnarray*}
&& \hspace{-0.3in} \int\limits_{\partial B} \big[ (A_1-A)\grad_x w(0) \cdot \nu \big] \overline{u}_\eps \, \, ds_y  -  \int\limits_{\widetilde D} \tilde A \grad_y w^{(1)} \cdot \grad_y \overline{u}_\eps + A_{min} w^{(1)} \overline{u}_\eps \, \, dy 
\\
&& \hspace{0.2in} =   \int\limits_{\partial B} \big[ (A_1-A)\grad_x w(0) \cdot \nu \big] \overline{u}_\eps \, \, ds_y  +  \int\limits_{\widetilde D} \overline{u}_\eps (\grad_y \cdot \widetilde A \grad_y w^{(1)} - A_{min} w^{(1)} ) \, \, dy \\
&& \hspace{0.7in} -  \int\limits_{\partial B} \left[ \left( A \frac{\partial  w^{(1)} }{\partial \nu_y} \right)^+ - \left( A_1 \frac{\partial  w^{(1)} }{\partial \nu_y} \right)^- \right] \overline{u}_\eps \, \, ds_y +  \int\limits_{\partial \widetilde{D}} A \frac{\partial  w^{(1)} }{\partial \nu_y}  \overline{u}_\eps \, \, ds_y .
 \end{eqnarray*}
 Now by using the boundary value problem for $ w^{(1)}$ we have that 
$$ \alpha \| u_\eps  \|_{H^1(\widetilde D)} ^2 \leq  \left| \, \,  \int\limits_{\partial \widetilde{D}} A \frac{\partial  w^{(1)} }{\partial \nu_y}  \overline{u}_\eps \, \, ds_y  \, \, \right| = \eps^{1-d} \left| \, \,  \int\limits_{\partial {D}} (A\grad_y  w^{(1)} (x/ \eps) \cdot \nu)  \overline{u}_\eps(x/ \eps) \, \, ds_x \, \,  \right| $$
$$ \hspace*{-3cm} \leq C\eps^{1-d} \|  \grad_y w^{(1)}(x/\eps) \|_{L^{\infty}(\partial D)} \| u_\eps (x/\eps) \|_{H^1( D)}.$$
 By the scaling we have that 
   \begin{eqnarray*}
&& \hspace{-0.2in}\| u_\eps  \|_{H^1(\widetilde D)} ^2 \leq   C\eps^{1-d/2} \|  \grad_y w^{(1)}(x/\eps) \|_{L^{\infty}(\partial D)} \| u_\eps (x/\eps) \|_{H^1( \widetilde D)}.
 \end{eqnarray*}
Since 
$$ \|  \grad_y w^{(1)}(x/\eps) \|_{L^{\infty}(\partial D)} = o(\eps^d) \quad \text{ for } \quad d=2,3$$
we can conclude that 
$$ \| u_\eps  \|_{H^1(\tilde D)} = o(\eps^2) \,\,\, \text{ for } \,\, d=2  \, \, \text{ and } \, \,  \| u_\eps  \|_{H^1(\tilde D)} = {o}(\eps^{5/2})  \, \, \text{ for } \,\, d=3,$$
which gives the result by scaling the norm back to the domain $D$. 
\end{proof}

By appealing to the triangle inequality we have the following result. 
\begin{corollary} \label{correctorconv}
Let $w^{(1)}$ be the solutions to \eqref{c2prob}, also assume that  $(w_\eps , v_\eps)$ and $(w,v)$ are defined by \eqref{eqwv} with $w$ being a smooth function then we have that 
 \begin{eqnarray}
\| w_\eps(x)-w(x) - \eps w(0) w^{(1)}(x/\eps) \|_{H^1(D)}  =\mathcal{O}(\eps^{d/2+1}). \label{error3}
 \end{eqnarray}
\end{corollary}

The arguments used in this section carry over to the case of multiple inhomogeneities. Indeed, for multiple inhomogeneities centered at $z_m$ with anisotropic material parameter $A_m$ we have that by using translation and summing over a finite number of inhomogeneities gives that the corrector takes the form
$$\widetilde{w}^{(1)}(x/\eps)=\sum\limits_{m=1}^M w(z_m) w^{(1)}_m(x/\eps)$$
where $w^{(1)}_m(x/\eps)$ is the solution to 
 \begin{eqnarray*}
\int\limits_{\R^d } \widetilde{A}_m \grad_y w_m^{(1)} \cdot \grad_y \overline{\varphi} + A_{min} w_m^{(1)} \overline{\varphi} \, \, dy= \int\limits_{\partial B_m} \big[ (A_m-A)\grad_x w(z_m) \cdot \nu \big] \overline{\varphi} \, \, ds_y  
\end{eqnarray*}
for all $\varphi \in H^1(\R^d)$ with $\widetilde{A}_m=A_m \, \chi_{B_m} + A(1-\chi_{B_m})$. The convergence results in this section still hold for $w(0) w^{(1)}(x/\eps)$ replaced by $\widetilde{w}^{(1)}(x/\eps)$.

\subsection*{Asymptotic Formulas}
Finally we have all the ingredients to give an asymptotic formula for the transmission eigenvalues using the results in \cite{sharinonlinear}. Note that we have assumed that contrast in the defect is only in the matrix valued material parameter (i.e. $n_\eps=n$ for all $\eps>0$), and we still take $A$ and $A_m$ constant matrices. Under this assumption we have that the operator ${\bf T}_{\eps}(\tau) = {\bf A}_{\eps}^{-1} {\bf B}_{0}+ \frac{1}{\tau} {\bf A}_{\eps}^{-1}{\bf C}$ converges in the operator norm.

We now recall Theorem 4.1 of \cite{sharinonlinear} which is a generalization of Osborn's Theorem (see \cite{osborn} for Osborn's result) to nonlinear eigenvalue problems. 
\begin{theorem}\label{shariresult}
Let $X$ be a Hilbert space and ${\bf T}_{\eps}(\tau) : X \to  X$ be a compact operator valued functions of $\tau$ which are
analytic in a region $U$ of the complex plane, such that $ \left\| {\bf T}_{\eps}(\tau)-{\bf T}_{0}(\tau) \right\| \rightarrow 0$ for all $\tau \in U$. Now assume that $\tau$ is a simple nonlinear eigenvalue of ${\bf T}_{0}(\tau)$ with normalized eigenfunction $\phi$. Then if 
$$\tau^2 \left( \frac{d}{d \tau} {\bf T}_{0}(\tau) \phi , \phi \right)  \neq - 1$$
we have that 
$$\hspace{-2.5in} \tau_\eps = \tau +\tau^2 \frac{ \big( ({\bf T}_{0}(\tau) - {\bf T}_{\eps}(\tau) )\phi , \phi \big)}{1+\tau^2 \left(\frac{d}{d \tau} {\bf T}_{0}(\tau) \phi , \phi \right) } $$
 $$\hspace{1in} + \mathcal{O} \left( \sup\limits_{\tau \in U} \left\| ({\bf T}_{\eps}(\tau)-{\bf T}_{0}(\tau)) \phi \right\|  \left\| ({\bf T}^*_{\eps}(\tau)-{\bf T}^*_{0}(\tau)) \phi \right\| \right)$$
with $\tau_\eps$ is a nonlinear eigenvalue for $ {\bf T}_{\eps}(\tau)$.
\end{theorem}
Theorem \ref{shariresult} only holds for simple eigenvalues. Notice that we have established the order of convergence of the operator defined by the transmission eigenvalue problem. In particular, the results in the previous section (see equation \eqref{error2}) gives that
$$ \left\| {\bf T}_{\eps}(\tau)(w_\tau , v_\tau)-{\bf T}_{0}(\tau) (w_\tau , v_\tau)\right\|  = \mathcal{O}(\eps^{d/2}).$$
We now consider the point wise convergence for the adjoint operator. 
\begin{lemma}
Let $(w_\tau , v_\tau) \in X(D)$ be the smooth eigenfunction corresponding to the eigenvalue $\tau$ of the operator ${\bf T}_{0}(\tau)$, then we have that 
$$ \left\| {\bf T}^*_{\eps}(\tau) (w_\tau , v_\tau) -{\bf T}^*_{0}(\tau) (w_\tau , v_\tau) \right\|  = \mathcal{O}(\eps^{d/2+1}).$$
\end{lemma}
\begin{proof}
Notice that ${\bf T}^*_{\eps}(\tau)= {\bf B}_{0}{\bf A}_{\eps}^{-1} + \frac{1}{\tau}{\bf C}{\bf A}_{\eps}^{-1} $ where we define $(w,v)={\bf A}_{0}^{-1}(w_\tau , v_\tau)$ and $(w_\eps,v_\eps)={\bf A}_{\eps}^{-1}(w_\tau , v_\tau)$. Now for any $(\varphi_1 , \varphi_2) \in X(D)$
\begin{align*}
\left( {\bf B}_{0}({\bf A}_{\eps}^{-1}- {\bf A}_{0}^{-1})(w_\tau , v_\tau) ; (\varphi_1 , \varphi_2) \right)=\mathcal{B}_0 \big((w_\eps-w, v_\eps-v) ; (\varphi_1 , \varphi_2) \big).
\end{align*}
Since the sesqulinear form $\mathcal{B}_0$ only has $L^2(D)$ terms, we have that 
\begin{align*}
\left| \big( {\bf B}_{0}({\bf A}_{\eps}^{-1}- {\bf A}_{0}^{-1})(w_\tau , v_\tau) ; (\varphi_1 , \varphi_2) \big) \right| \leq C \|(w_\eps-w, v_\eps-v)\|_{L^2(D)}  \|(\varphi_1 , \varphi_2)\|_{X(D)}
\end{align*}
By rescaling the $L^2$ norm in equation \eqref{unscaled} gives that 
$$\| w_\eps(x)-w(x) \|_{L^2(D)} + \| v_\eps(x)-v(x)  \|_{L^2(D)}=\mathcal{O}(\eps^{d/2+1}).$$
therefore $\left\| {\bf B}_{0}({\bf A}_{\eps}^{-1}- {\bf A}_{0}^{-1})(w_\tau , v_\tau) \right\|_{X(D)} = \mathcal{O}(\eps^{d/2+1})$. A similar argument gives that $\left\| {\bf C} ({\bf A}_{\eps}^{-1}- {\bf A}_{0}^{-1})(w_\tau , v_\tau) \right\|_{X(D)} = \mathcal{O}(\eps^{d/2+1})$, proving that claim. 
\end{proof}

\begin{remark}
This result shows why the case where $n_\eps \neq n$ can not be handled by this analytic framework. In particular, the rate of convergence in Theorem \ref{bconverge} for ${\bf B}_{0}-{\bf B}_{\eps}$ is not fast enough to provide an improved convergence rate for ${\bf T}^*_{\eps}(\tau)-{\bf T}^*_{0}(\tau) $ which is necessary to apply Theorem \ref{shariresult}. 
\end{remark}
We have just shown that the remainder term for the non-linear eigenvalue corrector formula is of the order $\eps^{d+1}$. 
To construct an asymptotic formula for the transmission eigenvalues we need to construct an asymptotic formula for 
$$ \Big(  {\bf T}_{0}(\tau) (w_\tau , v_\tau) - {\bf T}_{\eps}(\tau)(w_\tau , v_\tau); (w_\tau , v_\tau) \Big)_{X(D)}$$
where $(w_\tau , v_\tau)$ are the eigenfunctions for $\eps=0$. By equation \eqref{tedefectvarform} we have that $ {\bf B}_{0}(w,v) + \frac{1}{\tau} {\bf C} (w,v)=  \frac{1}{\tau}{\bf A}_{0}(w,v)$. Since the operator $ {\bf A}_{\eps}$ is self-adjoint for all $\eps \geq 0$ the definition of $ {\bf T}_{\eps}(\tau) $ in \eqref{tdef} gives that 
$$\hspace{-2in} \Big(  {\bf T}_{\eps}(\tau) (w_\tau , v_\tau) - {\bf T}_{0}(\tau)(w_\tau , v_\tau); (w_\tau , v_\tau) \Big)_{X(D)} =$$
$$ \hspace{2in} \frac{1}{\tau} \Big(  {\bf A}_{0} (w_\tau , v_\tau) ; \left({\bf A}_{\eps}^{-1} - {\bf A}_{0}^{-1} \right)(w_\tau , v_\tau) \Big)_{X(D)}.$$
This gives that we only need to construct an asymptotic formula for 
$$\mathcal{A}_0 \big( (w_\tau , v_\tau) ; \left({\bf A}_{\eps}^{-1} - {\bf A}_{0}^{-1} \right)(w_\tau , v_\tau) \big).$$
We now derive an asymptotic formula for ${\bf A}_{\eps}^{-1} - {\bf A}_{0}^{-1}$ with respect to the sesquilinear form $\mathcal{A}_0 \big( \cdot \, ; \cdot \big)$ which is given in the following result. 

\begin{theorem}
Let $(w_\tau , v_\tau)$ be the eigenfunctions for $\eps=0$ with transmission eigenvalue $\tau$ and define $(w,v)={\bf A}_{0}^{-1}(w_\tau , v_\tau)$, then we have that 
 \begin{eqnarray*}
&&\hspace{-0.3in}\mathcal{A}_0 \big( (w_\tau , v_\tau) ; \left({\bf A}_{\eps}^{-1} - {\bf A}_{0}^{-1} \right)(w_\tau , v_\tau) \big)= \eps^d \sum\limits_{m=1}^{M} (A-A_m)|B_m| \grad w_\tau(z_m) \cdot \grad \overline{w(z_m)} \\
&&\hspace{1in} +\eps^d \sum\limits_{m=1}^{M} w_\tau(z_m)  \overline{w(z_m)} \int\limits_{\partial B_m}   \left[ (A-A_m) \grad \overline{w^{(1)}_m}(y) \cdot \nu_y \right]\, ds_y+ o(\eps^d).
\end{eqnarray*}

\end{theorem}
\begin{proof}
We will prove the result for a single defect centered at the origin then by using translation and summing a finite number of such inhomogeneities, the asymptotic result follows. Letting $(w_\eps ,v_\eps )={\bf A}_{\eps}^{-1}(w_\tau , v_\tau)$, we have that 
 \begin{eqnarray}
\mathcal{A}_0 \big( (w_\tau , v_\tau) ; (w_\eps - w, v_\eps -v  \big) &=& ( \mathcal{A}_0- \mathcal{A}_\eps) \big( (w_\tau , v_\tau) ; (w_\eps , v_\eps ) \big) \nonumber \\ 
 												&=& ( \mathcal{A}_0- \mathcal{A}_\eps) \big( (w_\tau , v_\tau) ; (w_\eps - w -\eps w(0) w^{(1)} , v_\eps -v) \big) \nonumber \\
											&+& ( \mathcal{A}_0- \mathcal{A}_\eps )\big( (w_\tau , v_\tau) ; (w +\eps w(0) w^{(1)} , v )\big). \label{2termz} 
 \end{eqnarray}
 Recall, that by elliptic regularity we have the for any $\Omega$ such that $\eps B \subset \Omega \subset D$ the eigenfunctions are in $C^1(\Omega)$. Using this along with the support of $A-A_\eps$ and Corollary \ref{correctorconv} we can now estimate the first term 
 \begin{align*}
&\left| ( \mathcal{A}_0- \mathcal{A}_\eps ) \big( (w_\tau , v_\tau) ; (w_\eps - w -\eps w(0) w^{(1)} , v_\eps -v) \big) \right|=\\
&  \left|  \int\limits_{\eps B} (A-A_1)\grad w_\tau \cdot \grad \overline{w_\eps - w -\eps w(0) w^{(1)}} \, dx \right| \leq C || w_\tau ||_{H^1(\eps B)} || w_\eps - w -\eps w(0) w^{(1)} ||_{H^1(D)} \\
															&\qquad \qquad  \leq C \eps^{d/2+1} ||\chi_{\eps B}||_{L^2(D)} || w_\tau ||_{C^1(\Omega)} \leq C \eps^{d+1} || w_\tau ||_{C^1(\Omega)}.
\end{align*}
We now consider the second term of \eqref{2termz} which is given by
 \begin{align*}
& ( \mathcal{A}_0- \mathcal{A}_\eps )\big( (w_\tau , v_\tau) ; (w +\eps w(0) w^{(1)} , v )\big) = \int\limits_{\eps B} (A-A_1)\grad w_\tau \cdot \grad \overline{ w +\eps w(0) w^{(1)}} \, dx \\
 										&\qquad = \int\limits_{\eps B } (A-A_1)\grad w_\tau \cdot \grad \overline{w} \, dx + \eps \overline{w(0)} \int\limits_{D_\eps} (A-A_1)\grad w_\tau \cdot \grad \overline{ w^{(1)}} \, dx\\
										& = \eps^d (A-A_1)|B| \grad w_\tau(0) \cdot \grad \overline{w(0)} + \eps \overline{ w(0)} \int\limits_{\eps B} (A-A_1)\grad w_\tau \cdot \grad \overline{w^{(1)}} \, dx  + o(\eps^d)
\end{align*}
where we have used Taylor's expansion about  the origin to estimate the first integral. Now by the divergence theorem we have that the volume integral involving the eigenfunction and the corrector is given by 
 \begin{eqnarray*}
&&\eps \int\limits_{\eps B} (A-A_1)\grad w_\tau \cdot \grad \overline{w^{(1)}} \, dx= \eps \int\limits_{\eps B} w_\tau(x) \grad \cdot (A-A_1)\grad \overline{w^{(1)}}(x/ \eps) \, dx \\
&&\hspace{1.5in} + \eps \int\limits_{\partial(\eps B)}  w_\tau(x) \left[ (A-A_1) \grad \overline{w^{(1)}}(x/\eps) \cdot \nu_x \right]\, ds_x. 
\end{eqnarray*}
Now by rescaling the second integral for $x=\eps y$ and using a Taylor's expansion we have that integration is given by 
 \begin{eqnarray*}
&&\eps \int\limits_{\eps B} (A-A_1)\grad w_\tau \cdot \grad \overline{w^{(1)}} \, dx= \eps^{d+1} \int\limits_{ B} w_\tau(\eps y) \grad \cdot (A-A_1)\grad \overline{w^{(1)}}(y) \, dy \\
&&\hspace{2in} + \eps^d w_{\tau}(0) \int\limits_{\partial B}   \left[ (A-A_1) \grad \overline{w^{(1)}}(y) \cdot \nu_y \right]\, ds_y +o(\eps^d)
\end{eqnarray*}
proving the result. 
\end{proof}

Now we have all we need for an asymptotic formula for simple transmission eigenvalues. Notice that $\frac{d}{d \tau} {\bf T}_{0}(\tau) = -\frac{1}{\tau^2} {\bf A}_{0}^{-1} {\bf C}$, therefore we have that 
$$\tau^2 \left( \frac{d}{d \tau} {\bf T}_{0}(\tau) (w_\tau , v_\tau ) , (w_\tau , v_\tau ) \right) =- \mathcal{C}\big( (w_\tau , v_\tau) ; {\bf A}_{0}^{-1} (w_\tau , v_\tau ) \big).$$
For convenience let the constant
\begin{eqnarray}
q_m=  \int\limits_{\partial B_m}   \left[ (A-A_m) \grad \overline{w^{(1)}_m}(y) \cdot \nu_y \right]\, ds_y \label{polar-constant}
\end{eqnarray}
Therefore we have that simple transmission eigenvalues have the expansion.

\begin{theorem}\label{te-asymformula}
Let $(w_\tau , v_\tau)$ be the eigenfunctions for $\eps=0$ with simple transmission eigenvalue $\tau$ and define $(w,v)={\bf A}_{0}^{-1}(w_\tau , v_\tau)$, then we have that 
$$ \tau_\eps= \tau + \tau \eps^d \sum\limits_{m=1}^{M} \frac{ (A_m-A)|B_m| \grad w_\tau(z_m) \cdot \grad \overline{w(z_m)} + q_m  w_\tau(z_m) \overline{w(z_m)} }{ 1- \mathcal{C}\big( (w_\tau , v_\tau) ; (w , v )\big)  } +o(\eps^d)$$
where $q_m$ is given by \eqref{polar-constant} and
$$\mathcal{C} \big( (w_\tau , v_\tau) ; (w , v )\big)=\int\limits_{D} A_{min} w_\tau \overline{w} \,  - v_\tau \overline{v} \, dx.$$
\end{theorem}

\subsection*{Numerical Validation of the Asymptotic Formula}
The asymptotic formula given in Theorem \ref{te-asymformula} can potentially be used to determine the strength of the small defective region(s).  Notice that the MUSIC algorithm gives the location of the defect(s) and recall that the transmission eigenvalues for the perturbed media $\tau_\eps$ can be measured from the same scattering data needed for MUSIC but for a range of wave numbers $k$ (see \cite{chesnel}, \cite{p1}, \cite{harris}),  whereas the transmission eigenvalues $\tau$ and eigenfunctions $(w_\tau , v_\tau)$ for the unperturbed media can be computed since $A$ and $n$ are assumed to be known. In particular, denoting by ${\mathbf F}_\epsilon:=\left(u_\epsilon(\hat x_i, x_j,k)\right)_{i,j=1..N}$, the far field matrix due to the inhomogeneity $D$ with perturbation $D_\epsilon$, where we indicate its dependence on $k$, to determine the $\tau_\epsilon:=k_\epsilon^2$  we solve the regularized equation 
$$(\alpha+{\mathbf F}_\epsilon^*{\mathbf F}_\epsilon){\mathbf g}(k)={\mathbf F}_\epsilon^*(e^{ikz\cdot\hat x_i})_{i=1..N}, \qquad z\in D$$
for ${\mathbf g}(k)$ for a range of $k$. The transmission eigenvalues are those values of $k$ for which  $\|{\mathbf g}(k)\|_{\ell^2}$ blows up. To compute the transmission eigenvalues  for the unperturbed media  we use a continuous finite element method with the eigenvalue searching technique described in \cite{harris}, \cite{sun} and \cite{sun2}). In order to use the asymptotic formula in Theorem \ref{te-asymformula} one also needs the functions $(w,v)={\bf A}_{0}^{-1}(w_\tau , v_\tau)$ which can be solved for (e.g. using the FEM) since $A$ and $(w_\tau , v_\tau)$ are known. Having identified the location of the defect(s) from the MUSIC algorithm (i.e. the points $z_m$ are known) one can determine the strength of the defect(s) which is given by $(A_m-A)|B_m|$ and $q_m$ from the knowledge of two transmission eigenvalues. Notice, that the strength of the defect(s) only depend on the constitutive coefficients and geometry of the defect(s). 

We first consider a few examples to illustrate the convergence of the transmission eigenvalues as $\eps \rightarrow 0$ in $\R^2$. To do so, we denote the transmission eigenvalues for the unperturbed media by $k^2_j$ and the first transmission eigenvalue for the perturbed media by $k^2_j(\eps)$. To test our asymptotic formula we will check the order of convergence for two transmission eigenvalues. We compute the error and estimated order of convergence by 
$$\text{E}_j(\eps) = \big| k^2_j - k^2_j(\eps) \big| \quad \text{ and } \quad \text{EOC}_j = \log \Big( \text{E}_j(\eps) / \text{E}_j(\eps/2) \Big) / \log(2).$$
In our calculations we see that the order of convergence seems to be approximately second order which is what is predicted by Theorem \ref{te-asymformula}

\noindent {\bf Example 1.} Here we let $D=[-1,1]^2$ where $D_0$ is given by two disks of radius $\eps$ centered at $(0.25 ,0)$ and $(- 0.25 , - 0.25)$. For this case we take $n=n_\eps=1$ for all $\eps$ with 
$$A=\begin{pmatrix}  10 & 1 \\ 1  &  10  \end{pmatrix}$$ 
and $A_\eps =2 I$. Below in Table \ref{TE-conv-1} we show estimated order of convergence for two transmission eigenvalues. 

\begin{table}[!ht]
\centering
\begin{tabular}{c|cc}
$\eps$    & $\mathrm{EOC}_{1}$ & $\mathrm{EOC}_{2}$ \\
 \hline
  1/4  & $-$   & $-$ \\
 1/8  & 2.4423   & 0.8252 \\
 1/16 & 2.3365  &2.0673 \\
 1/32 & 2.0881 & 2.0001\\
1/64 & 2.1705  & 2.1549 \\
 \hline
\end{tabular}
\caption{\label{TE-conv-1} The estimated order of convergence for two eigenvalues where $D_0$ is two disks of radius $\eps$ centered at $(0.25 ,0)$ and $(- 0.25 , - 0.25)$.}
\end{table}

\noindent {\bf Example 2.} Here we let ${\displaystyle D= \left\{ (x_1 , x_2) \in \R^2 \,  : \, \frac{x_1^2}{4}+x_2^2 <1 \right\} }$ where $D_0$ is the disk centered at the origin of radius $\eps$. For this case we take $n=n_\eps=1$ for all $\eps$ with $A=10I$  and $A_\eps =2 I$. Below in Table \ref{TE-conv-2} we show estimated order of convergence for two transmission eigenvalues. 

\begin{table}[!ht]
\centering
\begin{tabular}{c|cc}
$\eps$    & $\mathrm{EOC}_{1}$ & $\mathrm{EOC}_{2}$ \\
 \hline
 1/4  & $-$   & $-$ \\
 1/8  & 1.9304   & 2.2957 \\
 1/16 & 2.1519  & 2.2278 \\
 1/32 & 2.1161 & 2.0304\\
1/64 & 3.1701  & 2.5851 \\
 \hline
\end{tabular}
\caption{\label{TE-conv-2} The estimated order of convergence for two eigenvalues where $D_0$ is the  disk of radius $\eps$ centered at the origin.}
\end{table}

Next we provide an example on how to use the asymptotic formula in Theorem \ref{te-asymformula} to obtain information about the strength of the perturbation. For the case of a homogeneous isotropic unit disc $D$ with  $A= \alpha I$ with $\alpha$ a positive constant and $n=1$ we have that the radially symmetric eigenfunctions corresponding to the eigenvalues $k^2$ are given by 
$$w_k(r) = \text{J}_0(k) \text{J}_0 \left( {k}/{\sqrt{\alpha} } \, r \right) \quad \text{and} \quad v_k(r) = \text{J}_0  \left( {k}/{\sqrt{\alpha} } \right) \text{J}_0 ({k}r).$$
Here $\text{J}_0$ is the first kind Bessel function of order zero.Using  the variational formulation,  the solution to ${\bf A}_{0}(w,v)=(w_k , v_k)$ can be show to be  
$$w(r) = c_1 \text{I}_0(r)  + \alpha^{-1} w_k(r)  \quad \text{and} \quad v(r) = c_2 \text{I}_0 (r)  -v_k (r)$$
where the constants $c_1$ and $c_2$ satisfy 
$$\left[
\begin{array}{cc} \text{I}_0(1) & \, \,  - \text{I}_0(1) \\  \alpha \text{I}'_0 (1) & \, \,   -\text{I}'_0 (1) \end{array}
\right]  
\left[
\begin{array}{c} c_1 \\ c_2 \end{array}
\right]  
=
\left[
\begin{array}{c}-(\alpha^{-1}+1)w_k(1)   \\ (1-\alpha) w'_k(1)  \end{array}
\right] $$
where here $\text{I}_0$ is the third kind Bessel function of order zero. This implies that the corrector term in Theorem \ref{te-asymformula} is known up to the weighted contrast $(A_m-A)|B_m|$ and `polarization' constant $q_m$. Assuming that two transmission eigenvalues for the unperturbed media are known and the corresponding two eigenvalues for the perturbed media is computed via the far-field data, by ignoring the $o(\eps^d)$ term in the asymptotic formula given in Theorem \ref{te-asymformula} one obtains a 2$\times$2 linear system of equations to determine the weighted contrast and polarization constant. For proof of concept, we consider the case where $D$ is the unit disk with $D_0$ being the disk of radius $\eps$ centered at $(0.25 ,0)$ where coefficients are taken be $A=10I$  and $A_\eps =2 I$. We assume that for this configuration the center of the disc is reconstructed using MUSIC as shown in Figure \ref{musicfig33}.
\begin{figure}[!ht]
\centering
\includegraphics[scale=0.3]{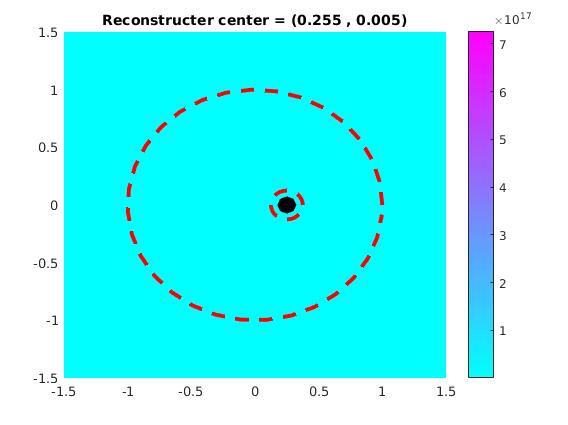}\includegraphics[scale=0.3]{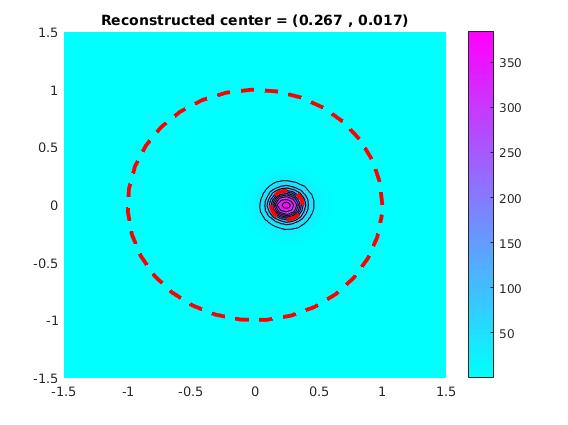}\\
\caption{Reconstruction  the defective region $D_\epsilon$ centered at $(0.25 ,0)$.  The figure on the left shows the reconstruction without noise, where the estimated location is $(0.255, 0.005)$. The figure on the right shows  the reconstruction with 10$\%$ noise, where the estimated location is $(0.267, 0.017)$. Here $N=20$.}
\label{musicfig33}
\end{figure}
Using the first two radially symmetric eigenvalues and functions we wish to determine the contrast for the particular case of $\eps = 1/2$. In this example the contrast is given by $-8$ and solving the 2$\times$2 linear system derived from the asymptotic formula recovers a contrast of $-7.3465$ if the exact location is used in the formula. Using the reconstructed center $(0.255, 0.005)$ for the case without noise in Figure \ref{musicfig33} we obtain   that the contrast is  $-7.1222$, and using the reconstructed center $(0.267, 0.017)$  for the case with $10\%$ noise in Figure \ref{musicfig33} we obtain   that the contrast is  $-6.6147$.  This preliminary example shows that one can determine information about the location and material properties of the small defects from a knowledge of the far-field data. Of course further investigation is needed to numerically validate our imaging method.
\section*{Acknowledgments}

{The research of F. Cakoni is supported in part by AFOSR grant  FA9550-17-1-0147, NSF Grant DMS-1602802 and Simons Foundation Award 392261.  The research of S. Moskow is supported in part by NSF Grant DMS-1411721.}


\end{document}